\documentclass[11pt,draft]{article}

\usepackage{amsmath}

\allowdisplaybreaks[1]
\usepackage{amsfonts,amsmath,amssymb,amsthm}

\theoremstyle{plain}
\newtheorem{theorem}{Theorem}

\theoremstyle{definition}
\newtheorem{definition}{Definition}

\newtheorem{proposition}{Proposition}

\newtheorem{remark}{Remark}

\newtheorem{lemma}[theorem]{Lemma}
\newtheorem{hypothesis}[theorem]{Hypothesis}

\def\sqr#1#2{{\vcenter{\vbox{\hrule height .#2pt \hbox{\vrule
 width .#2pt height#1pt \kern#1pt \vrule
width .#2pt} \hrule height .#2pt}}}}

\def\ds{\begin{displaystyle}}
\def\eds{\end{displaystyle}}
\def\dis{\displaystyle }
\def\<{\langle }
\def\>{\rangle }

\def\R{\mathbb R}

\def\Z{\mathbb Z}

\def\E{\mathbb E}
\def\P{\mathbb P}

\def\K{\mathbb K}

\title{Dissipative backward stochastic differential equations with locally Lipschitz nonlinearity.}

\author{Fulvia Confortola\\[.3em]
\normalsize\it
Dipartimento di Matematica, Politecnico di Milano\\
\normalsize\it
piazza Leonardo da Vinci 32, 20133 Milano, Italy\\
\normalsize\tt
fulvia.confortola@polimi.it}

\begin{document}

\maketitle

\begin{abstract} In this paper we study a class of backward stochastic differential
equations (BSDEs) of the form
$$
dY_t= -AY_tdt  -f_0(t,Y_t)dt -f_1(t,Y_t,Z_t)dt + Z_tdW_t ,\, 0 \leq t \leq  T; \,
Y_T= \xi
$$
in an infinite dimensional Hilbert space $H$, where the unbounded
operator $A$ is sectorial and dissipative and the nonlinearity $f_0(t,y)$ is dissipative and
defined for $y$ only taking values in a subspace of $H$. A typical
example is provided by the so-called polynomial nonlinearities. Applications are given to stochastic partial differential equations and spin systems.
\end{abstract}




\section{Introduction}

Let $H,K$ be real separable Hilbert spaces with norms $| \cdot|_H$ and $| \cdot|_K$. Let $W$ be a cylindrical Wiener process in $K$ defined on a
probability space $(\Omega,\mathcal{F},\P)$ and let $\{\mathcal{F}_t\}_{t \in
 [0,T]}$ denote its natural augmented
filtration. Let $\mathcal{L}^2(K,H)$ be the Hilbert space of Hilbert-Schmidt operators from $K$ to $H$.

We are interested in solving the following backward stochastic
differential equation
\begin{equation} \label{eq.int.intro}dY_t= -AY_tdt  - f(t,Y_t,Z_t)dt + Z_tdW_t
,\quad 0 \leq t \leq T, \quad Y_T= \xi
\end{equation}
where $\xi$ is a random variable with values in $H$, $f(t,Y_t,Z_t)=
f_0(t,Y_t) + f_1(t,Y_t,Z_t)$ and $ f_0,f_1$ are given
functions, and the operator $A$ is an unbounded operator with
domain $D(A)$ contained in $H$. The unknowns are the
processes $\{Y_t \}_{t \in
 [0,T]}$ and $\{ Z_t \}_{t \in
 [0,T]}$, which are required to be adapted with respect to the
 filtration of the Wiener process and take values in $H$, $\mathcal{L}^2(K,H)$ respectively.

In finite dimensional framework such type of equations
has been solved by Pardoux and Peng \cite{PPe1} in the nonlinear case. They
proved an existence and uniqueness result for the solution of the equation (\ref{eq.int.intro})
 when $A=0$, the
coefficient $f(t,y,z)$ is Lipschitz continuous in both variables
$y$ and $z$, and the data $\xi$ and the  process $\{f(t,0,0)\}_{t
\in [0,T]}$ are square integrable. Since this first result, many
papers were devoted to existence and uniqueness results under
weaker assumptions. In finite dimension, when $A=0$, the Lipschitz
condition on the coefficient $f$ with respect to the variable $y$
is replaced by a monotonicity assumption; moreover, more general
growth conditions in the variable $y$ are formulated. Let us
mention the contribution of Briand and Carmona \cite{BC}, for a
study of polynomial growth in $L^p$ with $p>2$, and the work of
Pardoux \cite {P3} for an arbitrary growth. In \cite{PR1} Pardoux
and Rascanu deal with a BSDE involving the subdifferential of a
convex function; in particular, one coefficient is  not everywhere
defined for $y$ in ${\R}^k$.

In other works the existence of the solution is proved when the data, $\xi$ and the
process $\{f(t,0,0)\}_{t \in [0,T]}$, are in $L^p$ for $p \in (1,2)$. El Karoui, Peng
and Quenez \cite{EKPQ} treat the case when $f$ is Lipschitz continuous; in \cite{BDHPS}
this result is generalized to the case of a monotone coefficient $f$ (both for equations
on a fixed and on a random time interval) and is studied even the case $p=1$.

In the infinite-dimensional framework Hu and Peng \cite{HP}, and
Oksendal and Zhang \cite{OZ} give an existence and uniqueness
result for the equation with an operator $A$, infinitesimal generator of a strongly continous semigroup and the
coefficient $f$ Lipschitz in $y$ and $z$. Pardoux
and Rascanu \cite{PR2} replace the operator $A$ with
the
subdifferential of a convex function and assume that $f$ is
dissipative, everywhere defined and continuous with respect to $y$, Lipschitz with
respect to $z$ and with linear growth in $y$ and $z$.

Special results deal with stochastic backward partial differential equations
(BSPDEs): we recall in particular the works of Ma and Yong \cite{MY1} and \cite{MY2}.
Earlier, Peng \cite{Pe} studied a backward stochastic partial
differential equation and regarded the classical Hamilton-Jacobi-Bellman equation of optimal stochastic control as
special case of this problem.

Our work extends these results in a special direction. We consider
an operator $A$ which is the generator of an analytic contraction
semigroup on $H$ and a coefficient $f(t,y,z)$ of the form $
f_0(t,y)+ f_1(t,y,z)$. The coefficient $f_1(t,y,z)$ is assumed to
be bounded and Lipschitz with respect to $y$ and $z$. The term
$f_0(t,y)$ is defined for $y$ only taking values in a suitable
subspace $H_{\alpha}$ of $H$ and it satisfies the following growth
condition for some $1<\gamma <1/ \alpha$, $S \geq 0$, $\P$-a.s.

         $\quad \quad \quad \quad |f_0(t,y)|_H \leq S(1+ ||y||_{H_{\alpha}}^{\gamma}) \quad \forall
t \in [0,T], \quad \forall y \in H_{\alpha}.$

Following \cite{HP}, we understand the equation (\ref{eq.int.intro}) in the following integral
form
\begin{equation}\label{pb}
Y_t -\int_t^T e^{(s-t)A}[f_0(s,Y_s)+f_1(s,Y_s,Z_s)]ds + \int_t^T
e^{(s-t)A}Z_sdW_s= e^{(T-t)A}\xi,
\end{equation}
requiring, in particular, that $Y$ takes values in $H_{\alpha}$. This
requires generally that the final condition also takes values in the
smaller space $H_{\alpha}$.
We take as $H_{\alpha}$ a real interpolation space which belongs to the class
$J_{\alpha}$ between $H$ and the domain of an operator $A$ (see Section \ref{1}). Moreover $f_0(t, \cdot)$ is assumed to be locally Lipschitz from $H_{\alpha}$ into $H$ and dissipative in $H$.
We prove (Theorem \ref{cap4.T2}) that if $\xi$ takes its values in the closure of $D(A)$
in $H_{\alpha}$ and is such that $||\xi||_{H_{\alpha}}$ is essentially bounded, then
equation (\ref{pb}) has a unique mild solution, i.e. there exists a unique pair of
progressively measurable processes $Y: \Omega \times[0,T] \rightarrow H_{\alpha}$, $Z:
\Omega \times[0,T] \rightarrow {\mathcal{L}}^2(K;H)$, satisfying $\P$-a.s. equality
(\ref{pb}) for every $t$ in $[0,T]$ and such that $\E
\sup_{t\in[0,T]}||Y_t||_{H_{\alpha}}^2+\E\int_0^T ||Z_t||_{\mathcal{L}^2(K,H)}^2dt< \infty.$

This result extends former results concerning the
deterministic case to the stochastic framework: see \cite{L2}, where previous works of
Fujita - Kato \cite{FK}, Pazy \cite{Pa} and others are collected. In these papers similar assumptions are made on the coefficients $f_0$, $f_1$ and on the operator $A$.

The plan of the paper is as follows. In Section \ref{1} some notations and definitions are
fixed. In Section \ref{2} existence and uniqueness of the solution of a simplified equation are proved,
  where $f_1$ is a bounded progressively
measurable process which does not depend on $y$ and $z$. In Section \ref{3}, applying the previous result, a fixed point argument is used in order to prove our main result on
existence and uniqueness of a mild solution of (\ref{pb}). Section \ref{4} is devoted
to applications.

\section{ Notations and setting }\label{1}

The letters $K$ and $H$ will always denote two real separable Hilbert spaces. Scalar product
is denoted by $\langle \cdot, \cdot\rangle$; $\mathcal{L}^2(K;H)$ is the separable Hilbert space of
Hilbert-Schmidt operators from $K$ to $H$ endowed with the Hilbert-Schmidt norm.
$W=\{W_t \}_{t \in [0,T]}$ is a cylindrical Wiener
process with values in $K$, defined on a complete probability space $(\Omega,\mathcal{F}, \P)$. $\{\mathcal{F}_t\}_{t \in [0,T]}$ is
the natural filtration of $W$, augmented with the family of
$\mathbb{P}$-null sets of $\mathcal{F}$.

Next we define several classes of stochastic processes with values
in a Banach space $X$.
\begin{itemize}
\item
$L^2(\Omega\times [0,T];X)$
denotes the space of measurable $X$-valued
processes $Y$ such that $\dis{\left[\E \int_0^T |Y_\tau|^2\,d\tau\right]^{1/2}}$ is finite, identified up to modification.

\item
$L^2(\Omega; C([0,T];X))$ denotes the space of continuous $X$-valued
processes $Y$ such that
$\dis{\left[\E \sup_{\tau\in [0,T]}|Y_\tau|^2 \right]^{1/2}}
$ is finite, identified up to indistinguishability.

\item
$C^{\alpha}([0,T];X)$ denotes the space of $\alpha$-H\"{o}lderian functions
 on $[0,T]$ with values in $X$ such that $[f]_{\alpha}=\dis{\sup_{0\leq x<y\leq T}
 \frac{|f(x)-f(y)|}{(y-x)^{\alpha}}<\infty}$.

\end{itemize}

Now we need to  recall several preliminaries on semigroup and interpolation spaces. We refer the reader to \cite{L2} for the proofs and other related results.

A linear operator $A$ in a Banach space $X$, with domain $D(A) \subset X$, is called sectorial if there are constants $\omega \in \mathbb{R}$, $
\theta \in (\pi /2, \pi)$, $M>0$ such that
\begin{equation} \label{oper.sett}\left\{%
\begin{array}{ll}
    (i)  \quad \rho(A)
    \supseteq S_{\theta,\omega}=
    \{\lambda \in\mathbb{C}:\lambda\neq\omega, |arg(\lambda - \omega)|< \theta \}, \\
    (ii) \quad ||(\lambda I -A)^{-1}||_{\mathcal{L}(X)} \leq \frac{M}{|\lambda - \omega|}
    \quad \forall \lambda \in S_{\theta,\omega} \\
\end{array}%
\right.
\end{equation}
where $\rho(A)$ is the resolvent set of $A$.
 For every $t>0$,
(\ref{oper.sett}) allows us to define a linear bounded operator $e^{tA}$ in $X$, by
means of the Dunford integral
\begin{equation} \label{int.dunf} e^{tA}=\frac{1}{2 \pi i}\int_{\omega +\gamma_{r, \eta}} e^{t
\lambda}(\lambda I -A)^{-1}d \lambda, \quad t>0, \end{equation} where, $r>0,
\eta \in ( \pi/2, \pi)$ and $\gamma_{r, \eta}$ is the curve $\{\lambda \in \mathbb{C}: |arg
\lambda|=\eta, |\lambda|\geq r \} \cup \{\lambda \in \mathbb{C} : |arg \lambda| \leq \eta,
|\lambda|= r \}$, oriented counterclockwise. We also set $
\label{op.set.0}e^{0A}x=x, \forall x \in X.$
Since the function $\lambda \mapsto e^{t \lambda}R(\lambda,A)$ is holomorphic in
$S_{\theta,\omega}$, the definition of $e^{tA}$ is independent of the choice of $r$ and
$\eta$.
If $A$ is sectorial, the function $[0, + \infty) \rightarrow L(X)$, $t \mapsto e^{tA}$,
with $e^{tA}$ defined by (\ref{int.dunf}) is called \emph{analytic
semigroup generated by A} in $X$.
We note that for every $x \in X$ the function $t \mapsto e^{tA}x$
is analytic (and hence continuous) for $t>0$. $e^{tA}$ is a strongly continuous semigroup if and only if
 $D(A)$ is dense in $X$; in particular this holds if $X$ is a
reflexive space.

We need to introduce suitable classes of subspaces of $X$.
 \begin{definition}\label{DAap}Let $(\alpha,p)$ be two numbers such that $0< \alpha
 <1$, $
 1 \leq p\leq \infty$ or $(\alpha,p)=(1,
\infty)$. Then we denote with $D_A(\alpha,p)$ the space

$\quad \quad \quad \quad D_A(\alpha,p)= \{x \in X: t \mapsto v(t)=||t^{1-\alpha- 1/p}Ae^{tA}x|| \in L^p(0,1) \}$

where $||x||_{D_A(\alpha, p)}= ||x||_X + [x]_{\alpha}= ||x||_X + ||v||_{L^p(0,1)}.$

(We set as usual $1/ \infty=0$).
\end{definition}

We recall here some estimates for the function $t\mapsto e^{tA}$
 when $t\rightarrow 0$, which we will use in the sequel. For convenience, in the next
 proposition we set $D_A(0,p )=X, \quad p \in [1, \infty].$

 \begin{proposition} \label{2.2.9lun} Let $(\alpha,p)$, $(\beta,p) \in (0,1) \times [1, +\infty] \cup \{(1, \infty)
 \}$, $\alpha \leq \beta$. Then there exists $C=C(p;
  \alpha,\beta)$ such that $$ ||t^{- \alpha + \beta}e^{tA}||_{L(D_A(\alpha,p),D_A(\beta,p))} \leq C, \quad 0<t \leq 1.$$
 \end{proposition}

\begin{definition} Let $0 \leq \alpha \leq 1$ and let $D,X$ be Banach spaces, $D \subset X$. A Banach space $Y$
such that $D \subset Y \subset X$ is said to belong to the class $J_{\alpha}$ between
$X$ and $D$ if there is a constant $C$ such that $||x||_Y \leq C ||x||_X^{1- \alpha}||x||_D^{\alpha}, \quad
\forall x \in D.$ In this case we write $Y \in J_{\alpha}(X,D)$.
\end{definition}

Now we give the definition of solution to the BSDE:
\begin{equation}\label{pbS1}
Y_t -\int_t^T e^{(s-t)A}[f_0(s,Y_s)+f_1(s,Y_s,Z_s)]ds + \int_t^T
e^{(s-t)A}Z_sdW_s= e^{(T-t)A}\xi,
\end{equation}

 \begin{definition} A pair of progressively measurable  processes $(Y,Z)$ is called mild
 solution of (\ref{pbS1}) if it belongs to the space $L^2(\Omega;C([0, T];H_{\alpha}))\times L^2(\Omega \times
[0,T];\mathcal{L}^2(K,H))$ and $\P$-a.s.solves the integral
equation (\ref{pbS1}) on the interval $[0,T]$.
\end{definition}

We finally state a lemma needed in the sequel. It
is a generalization of the well known Gronwall's lemma. Its proof is given in the Appendix.
\begin{lemma}\label{Gronwall2} Assume $a,b,\alpha, \beta$ are nonnegative constants, with
$\alpha <1$, $\beta>0$ and $0<T <\infty$. For any nonnegative process $U \in L^1(\Omega
\times[0,T]) $, satisfying $\P$-a.s. $ U_t \leq a(T-t)^{-\alpha} +b \int_t^T (s-t)^{\beta - 1}\E^{\mathcal{F}_t}U_sds$
for almost every $t \in [0,T]$, it holds $\P$-a.s.
$ U_t \leq aM(T-t)^{- \alpha}, $
for almost every  $ t \in [0, T]$. $M$ is a constant depending only on $b, \alpha,
\beta,T$.
\end{lemma}

\section{A simplified equation}\label{2}
As a preparation for the study of (\ref{pb}), in this section we consider the following simplified
version of that equation:
\begin{equation} \label{e}Y_t -\int_t^T e^{(s-t)A}[f_0(s,Y_s)ds+f_1(s)]ds + \int_t^T e^{(s-t)A}Z_sdW_s=
e^{(T-t)A}\xi, 
  \end{equation}
for all $t \in
 [0,T]$.

We suppose that the following assumptions hold.

\begin{hypothesis}\label{H1}
\end{hypothesis}
\begin{description}

 \item [1.]
 $A :D(A)\subset H \rightarrow H$ is a sectorial operator. We also assume that $A$ is dissipative, i.e. it satisfies $<Ay,y>\leq  0, \forall y \in D(A)$;

 \item [2.] for some $0<  \alpha <1$ there exists a
Banach space $H_{\alpha}$ continuously embedded in $H$ and such that

(i)  $\quad D_A(\alpha,1)\subset H_{\alpha} \subset D_A(\alpha,\infty)$;

(ii) $\quad \mbox{the part of } A \mbox{ in }  H_{\alpha} \mbox{ is sectorial in } H_{\alpha}.$

\item [3.]the final condition $\xi$ is an $\mathcal{F}_T$-measurable random variable defined on $\Omega$ with values in the closure of $D(A)$ with respect to $H_{\alpha}$-norm. We denote this set $\overline{D(A)}^{H_{\alpha}}$. Moreover $\xi$ belongs to
$L^{\infty}(\Omega,\mathcal{F}_T,\P;H_{\alpha})$;

 \item [4.] $f_0:\Omega \times[0,T] \times H_{\alpha} \rightarrow H$
satisfies:
\begin{description}
\item{i)} $\{f_0(t,y)\}_{t\in [0,T]}$ is progressively measurable $\forall y
\in H_{\alpha}$;
  \item{ii)} there exist constants $S >0$, $1< \gamma< 1/ \alpha$ such that  $\P$-a.s.

  $\quad \quad \quad |f_0(t,y)|_H\leq  S(1+ ||y||_{H_{\alpha}}^{\gamma}) \quad t \in [0,T], y \in H_{\alpha};
$

 \item{iii)} for every $R>0$ there is $L_R>0 $ such that $\P$-a.s.

$\quad \quad \quad  |f_0(t,y_1)-f_0(t,y_2)|_H \leq L_R ||y_1-y_2||_{H_{\alpha}}$

for $t \in [0,T]$ and $y_i \in H_{\alpha} $ with $||y_i||_{H_{\alpha}} \leq R$;
    \item{iv)} there exists a number
    $\mu \in \mathbb{R}$ such that $\P\text{-a.s.}$, $ \forall t \in [0,T]$, $y_1,y_2 \in H_{\alpha}$,
\begin{equation} \label{mu.dis} \, <f_0(t,y_1)-f_0(t,y_2),y_1-y_2>_H \leq \mu |y_1-y_2|_H^2;
\end{equation}
\end{description}
\item [5.]$f_1: \Omega\times [0,T] \rightarrow H$ is progressively
measurable and for some constant $C>0$ it satisfies $\P$-a.s. $|f_1(t)|_H \leq C,$
for $t \in [0,T]$.
\end{description}

\begin{remark} \label{rem.mu} \rm We note that the pair $(Y,Z)$ solves the BSDE
(\ref{e}) with final condition $\xi$ and drift $f=f_0 + f_1$ if and only if the pair
$(\bar{Y},\bar{Z}):= (e^{\lambda t}Y_t,e^{\lambda t}Z_t)$ is a solution of the same
equation with final condition $e^{\lambda T} \xi$ and drift $f'(t,y):= f_0^{'}(t,y)+
f_1^{'}(t)$ where $ f_0^{'}(t,y)= e^{\lambda t}(f_0(t,e^{-\lambda t}y) - \lambda y)$, $f_1^{'}(t)= e^{\lambda t}f_1(t).$ If we choose $\mu=\lambda$, then
$f_0^{'}$ satisfies the same assumption as $f_0$, but with (\ref{mu.dis}) replaced by
$ <f_0(t,y_1)-f_0(t,y_2),y_1-y_2>_H \leq 0.$
If this last condition holds, then $f_0$ is called dissipative. Hence, without loss
of generality, we shall assume until the end that $f_0$ is dissipative, or equivalently
that $\mu=0$ in (\ref{mu.dis}).
\end{remark}

\subsection{A priori estimates}

 We prove a basic estimate for the solution in the norm of $H$.

\begin{proposition} \label{eH} Suppose that Hypothesis \ref{H1} holds;
if $(Y,Z) $ is a mild solution of (\ref{e}) on the interval $[a,T]$, $0 \leq  a \leq T$, then
there exists a constant $C_1$, which depends only on $||\xi||_{L^{\infty}( \Omega;H)}$
and on the constants S of 4.ii) and C of 5. such that
$\P$-a.s. $ \sup_{a \leq
t \leq T} ||Y_t||_H \leq C_1.$ In particular the constant $C_1$ is independent of $a$.
\end{proposition}

\begin{proof}
\rm Let the pair $(Y,Z)\in L^2( \Omega, C([a,T];H_{\alpha}) \times
L^2(\Omega \times [a,T];\mathcal{L}^2(K;H))$ satisfy (\ref{e}).
Let us introduce the  operators $J_n=n(nI-A)^{-1}$,
$n>0$. We note that
the operators $AJ_n$ are the Yosida approximations
of $A$ and they are bounded. Moreover $|J_nx-x| \rightarrow 0$ as $n \rightarrow \infty$, for every $x \in H$.
We set
$Y_t^n=J_nY_t$, $Z_t^n=J_nZ_t$.
It is readily verified that $Y^n$ admits the It\^o differential

$dY_t^n= -AY_t^ndt -J_nf(t,Y_t)dt -J_nf_1(t)dt+Z_t^ndW_t, \mbox{and} \quad Y_T^n=J_n\xi.$

Applying the Ito formula to $|Y_t^n|_H^2$, using the dissipativity of $A$, we obtain
\begin{equation} \label{Jneq}
\begin{array}{l}
\dis{ |Y_t^n|_H^2 + \int_t^T||Z_s^n ||_{\mathcal{L}^2(K;H)}^2ds} \leq  \dis{|J_n\xi|_H^2+
 2 \int_t^T <J_nf_0(s,Y_s),Y_s^n>_Hds} +\\
 \quad \quad +  \dis{2 \int_t^T <J_nf_1(s),Y_s^n>_Hds-  2 \int_t^T<Y_s^n,Z_s^n dW_s>_H}.
\end{array}
\end{equation}
We note that
$\int_t^T <J_nf_0(s,Y_s)+J_nf_1(s),Y_s^n>_Hds \rightarrow \int_t^T <f_0(s,Y_s)+f_1(s),Y_s>_Hds$ by dominated convergence, as $n \rightarrow
\infty$.
Moreover by the
dominated convergence theorem we have $\int_t^T||(Z_s^n)^*Y_s^n
-Z_s^*Y_s||_K^2ds \rightarrow 0 \quad \P\mbox{-a.s.}
$ and it follows that  $\int_t^T<Y_s^n,Z_s^ndW_s>_H \rightarrow \int_t^T<Y_s,Z_sdW_s>_H$
in probability. If we let $n \rightarrow \infty$ in (\ref{Jneq}) we obtain

 $\dis{|Y_t|_H^2 + \int_t^T||Z_s ||_{\mathcal{L}^2(K;H)}^2ds } \leq  \dis{ |\xi|_H^2+
 2 \int_t^T <f_0(s,Y_s)+f_1(s),Y_s>_Hds} $

  $\quad \quad \quad \quad \quad \quad \quad \quad \quad \quad \quad  \quad\dis{-  2 \int_t^T<Y_s,Z_s dW_s>_H.}$

Recalling (\ref{mu.dis}), that we assume to hold with $\mu=0$, it follows that
$$
\begin{array}{l}
\dis{|Y_t|_H^2 +\int_t^T ||Z_s||_{{\mathcal{L}}^2(K,H)}^2} \leq \\
\quad  \quad \leq \dis{ |\xi|_H^2+
 2 \int_t^T <f_0(s,0),Y_s>_H + 2 \int_t^T <f_1(s),Y_s>_Hds +}\\
 \quad \quad \quad \dis{- 2 \int_t^T<Y_s,Z_s dW_s>_H} \\
\quad \quad \leq  \dis{|\xi|_H^2+
  \int_t^T |f(s,0)|_H^2ds +\int_t^T |f_1(s)|_H^2ds+ 2\int_t^T |Y_s|_H^2ds +}\\
  \quad\quad\quad \dis{-  2 \int_t^T<Y_s,Z_s
  dW_s>_H}.
 \end{array}
 $$

Now, since
$\sup_{0 \leq t \leq T}|f(t,0)|_H^2 \leq S^2$
and since the stochastic integral $  \int_{a}^t <Y_s,Z_s dW_s>_H$, $ t \in [a ,T]$ is a
martingale, if we take the conditional expectation given $ \mathcal{F}_t$ we have$$
\begin{array}{l}
\dis{|Y_t|_H^2 \leq \E^{\mathcal{F}_t}|\xi|_H^2 +
2\E^{\mathcal{F}_t}\int_t^T |Y_s|_H^2ds +}\\
\quad \quad \quad \quad\dis{+ \E^{\mathcal{F}_t} \int_t^T |f(s,0)|_H^2ds +
\E^{\mathcal{F}_t} \int_t^T |f_1(s)|_H^2ds}\\
\quad \quad \quad \dis{\leq |\xi|_{L^{\infty}(\Omega,H)}^2+ (S^2 +C^2)T+
2\int_t^T \E^{\mathcal{F}_t}|Y_s|_H^2ds}.\\
\end{array}
$$

Since $Y$ belongs to $L^2(\Omega;C([a,T];H_{\alpha}))$ and,
consequently, $ ||Y||_{H_{\alpha}}^2 \in L^1(\Omega \times
[0,T])$, we can apply Lemma \ref{Gronwall2} to $|Y|_H^2$ and
conclude that

$\quad \quad \quad \quad \quad |Y_t|_H^2 \leq (|\xi|_{L^{\infty}(\Omega,H)}^2 +
[S^2+C^2]T)(1+2Te^{2T} ).$ \end{proof}

Now we will show that the result of Proposition \ref{eH}, together with the growth
condition satisfied by $f_0$, yields an a priori estimate on the solution in the
$H_{\alpha}$-norm.

Let $0< \alpha<1$ and let $\gamma
>1$ be given by 4.ii). We fix $\theta = \alpha \gamma$ and consider the  Banach space $D_A(\theta,\infty)$ introduced in Definition \ref{DAap}. It is easy to check (see \cite{L2}) that, if we take $\theta \in (0,1), \theta >\alpha$,
then $H_{\alpha}$ contains $D_A(\theta,\infty)$ and belongs to the class
$J_{\alpha/\theta}$ between $D_A(\theta,\infty)$ and $H$, hence the following inequality
is satisfied:
\begin{equation}\label{Ja/theta}
  |x|_{H_{\alpha}} \leq c |x|_{D_A(\theta,\infty)}^{\frac{\alpha}{\theta}}|x|_H^{1-
  \frac{\alpha}{\theta}}, \quad x \in D_A(\theta,\infty).
\end{equation}

\begin{proposition} \label{eHa}Suppose that Hypothesis \ref{H1} is satisfied.
Let $(Y,Z)$ be a mild solution of (\ref{e}) in $ [a,T]$, $a \geq 0$ and assume that
there exists two constants $R>0$ and $K>0$, possibly depending on $a$, such that, $\P$-a.s.,
\begin{equation} \label{K} \sup_{t \in [a,T]}||Y_t||_{H_{\alpha}} \leq R, \quad \sup_{t \in [a,T]}|Y_t|_H \leq K.
\end{equation}

Then the following inequality holds $\P$-a.s.:
\begin{equation}\label{Hastima}
|Y_t|_{L^{\infty}(\Omega, D_A(\theta,\infty))} \leq C_2 \frac{1}{(T-t)^{\theta -
\alpha}},\quad a \leq t <T
\end{equation}
with $C_2$ depending on the operator $A$, $||\xi||_{L^{\infty}(\Omega, H_{\alpha})}$,
$\theta$, $\alpha$, $K $, $C$ of 5. and $S$ of 4.ii) of Hypothesis 2.
\end{proposition}

\begin{proof} \rm Taking the conditional expectation given $\mathcal{F}_t$ in equation
(\ref{e}) we find

$\quad \quad Y_t = \E^{\mathcal{F}_t} \big(e^{(T-t)A}\xi + \int_t^T
e^{(s-t)A}[f_0(s,Y_s)+f_1(s)]ds\big), \quad a \leq t \leq T.$

Consequently, we have
\begin{equation}\label{stimaInfty}
\begin{array}{l}
\dis{||Y_t||_{D_A(\theta, \infty)} \leq \E^{\mathcal{F}_t}
||e^{(T-t)A}\xi||_{D_A(\theta, \infty)} }\\
 \quad \quad \quad \quad \quad \dis{+ \E^{\mathcal{F}_t}\int_t^T
||e^{(s-t)A}[f_0(s,Y_s)+f_1(s)]||_{D_A(\theta, \infty)}ds,\quad a \leq t \leq
T}.\\
\end{array}
\end{equation}

Since $H_{\alpha} \subset
D_A(\alpha, \infty)$, we have
\begin{equation}\begin{array}{l}\label{1.term}
 \dis{\E^{\mathcal{F}_t}||e^{(T-t)A} \xi||_{D_A(\theta, \infty)}} \leq \\
  \quad \leq \dis{\E^{\mathcal{F}_t} ||e^{(T-t)A}||_{L(D_A(\alpha,\infty),D_A(\theta,\infty))}
||\xi||_{L^{\infty}(\Omega,D_A(\alpha,\infty))}} \\
   \quad \leq  \dis{\frac{C_0}{(T-t)^{\theta -
  \alpha}}||\xi||_{L^{\infty}(\Omega,H_{\alpha})}},
\end{array}
\end{equation}
with $C_0= C_0( \alpha, \theta, \infty)$, where in the last inequality we use
Proposition \ref{2.2.9lun}.
Moreover
$$
\begin{array}{l}
\dis{\E^{\mathcal{F}_t}\int_t^T
||e^{(s-t)A}[f_0(s,Y_s)+f_1(s)]||_{D_A(\theta,
\infty)}ds } \leq \\
\leq \dis{ \E^{\mathcal{F}_t}\int_t^T ||e^{(s-t)A}||_{L(H,D_A(\theta,
\infty))}|f_0(s,Y_s)+f_1(s)|_H ds}\leq \\
\quad  \leq \dis{\E^{\mathcal{F}_t} \big(\int_t^T
\frac{C_1}{(s-t)^{\theta}}[|f_0(s,Y_s)|_H + |f_1(s)|_H ] ds\big)} \leq \\
 \quad  \leq
 \dis{\E^{\mathcal{F}_t}\int_t^T
\frac{C_1}{(s-t)^{\theta}}[ S(1+
||Y_s||_{H_{\alpha}}^{\gamma}) +C ]ds} .\\
\end{array}
$$
In the inequality we used Hypotheses 4.ii) and 5. and Proposition \ref{2.2.9lun}.
Recalling
(\ref{Ja/theta}), we conclude that the last term is
dominated by
$$ \begin{array}{l}
  \quad \quad
  \dis{\E^{\mathcal{F}_t}\int_t^T
\frac{C_1}{(s-t)^{\theta}}\left[S(1+ c|Y_s|_H^{\gamma(1- \alpha) /
\theta}||Y_s||_{D_A(\theta, \infty)}^{\gamma \alpha / \theta}) +C \right]ds} =\\
 \quad \quad = \dis{\E^{\mathcal{F}_t}\int_t^T \frac{C_1}{(s-t)^{\theta}}\left[S(1+
c|Y_s|_H^{\gamma(1- \alpha) / \theta}||Y_s||_{D_A(\theta, \infty)}) +C \right]ds,} \\
\end{array}
$$
by choosing $\theta= \alpha \gamma$. By the second inequality in (\ref{K}) this can be estimated by
 \begin{equation}\label{**}
 \begin{array}{l}
 \quad \quad   \dis{\int_t^T \frac{C_1}{(s-t)^{\theta}}S(1+ cK^{\gamma(1- \alpha) /
\theta}\E^{\mathcal{F}_t}||Y_s||_{D_A(\theta, \infty)} + C)ds} \\
 \quad  \leq \dis{ \int_t^T
  \frac{C_1}{(s-t)^{\theta}}(C+S)ds
  + \int_t^T \frac{C_1}{(s-t)^{\theta}}Sc K^{\gamma(1-
\alpha) / \theta}\E^{\mathcal{F}_t}||Y_s||_{D_A(\theta, \infty)}}ds.\\
\end{array}
\end{equation}

Hence by (\ref{1.term}) and (\ref{**}) it follows
$$
\begin{array}{l}
\dis{||Y_t||_{D_A(\theta,\infty)} \leq  \frac{C_0}{(T-t)^{\theta -
\alpha}}||\xi||_{L^{\infty}(\Omega,H_{\alpha})}
     + \int_t^T
  \frac{C_1}{(s-t)^{\theta}}(C+S)ds}\\
  \quad \quad   \quad \quad \quad  \dis{+ \int_t^T \frac{C_1}{(s-t)^{\theta}}Sc K^{\gamma(1-
\alpha) / \theta}\E^{\mathcal{F}_t}||Y_s||_{D_A(\theta, \infty)}ds},\\
\end{array}
$$
and (\ref{Hastima}) follows from Lemma \ref{Gronwall2}.
In order to justify the application of Lemma \ref{Gronwall2}, we need to prove that
$||Y||_{D_A(\theta, \infty)}$ belongs to $L^1(\Omega \times [a,T])$.
This also follows from(\ref{1.term}) and (\ref{**}) since, for some constant $K_1$,
$$
\begin{array}{l}
\dis{||Y_t||_{D_A(\theta,\infty)}} \leq \\
\quad \leq \dis{ \frac{K_1}{(T-t)^{\theta -
\alpha}}||\xi||_{L^{\infty}(\Omega,H_{\alpha})} + \E^{\mathcal{F}_t}[ \sup_{s
\in[a,T]}(1+||Y_s||_{H_{\alpha}}^{\gamma})
  \int_t^T \frac{ds}{(s-t)^{\theta}}]}\\
 \quad \leq \dis{\frac{K_1}{(T-t)^{\theta -
\alpha}}||\xi||_{L^{\infty}(\Omega,H_{\alpha})} +  (1+R^{\gamma})
  \int_t^T \frac{ds}{(s-t)^{\theta}}}.\\
\end{array}
$$
\end{proof}

\subsection{Local existence and uniqueness}

We prove that, under Hypothesis \ref{H1}, there exists a
unique solution of (\ref{e}) on an interval $[T-\delta,T]$ with $\delta$
sufficiently small.

To treat the ordinary integral in the left hand side of (\ref{e}), we need the
following result, whose proof can be found in \cite{L2}, Proposition 4.2.1 and Lemma 7.1.1.

\begin{lemma}\label{phi}

Let $\phi \in L^{\infty}((a,T);H),\quad 0<a<T$ and set
$$v(t)= \int_t^T e^{(s-t)A} \phi(s)ds, \quad a \leq t \leq T.$$
If $0<\alpha<1$, then  $v \in C^{1- \alpha}([a,T];D_A(\alpha,1))$ and there is $G_0 >0$,
not depending on $a$, such that
$$||v||_{C^{1- \alpha}([a,T];D_A(\alpha,1))}\leq G_0||\phi||_{
L^{\infty}((a,T);H)}.$$ Since $D_A(\alpha,1)\subset H_{\alpha}$, we also have $v \in C^{1- \alpha}([a,T];H_{\alpha})$ and there is $G>0$, not depending on $a$, such that
$$||v||_{C^{1- \alpha}([a,T];H_{\alpha})}\leq G||\phi||_{ L^{\infty}((a,T);H)}.$$
\end{lemma}

\begin{theorem} \label{el} Let us assume that Hypothesis \ref{H1} holds, except possibly 4.iv).
Then there exists $\delta>0$ such that the equation (\ref{e}) has a unique local mild solution
$(Y,Z) \in L^2(\Omega;C([T - \delta,T];H_{\alpha})) \times L^2(\Omega \times [T- \delta,
T];{\mathcal{L}}^2(K;H)) $.
\end{theorem}


 \begin{remark} \rm The
dissipativity condition 4.iv) only plays a role in obtaining the a priori
estimate in $H$ (Proposition \ref{eH}) and consequently global
existence, as we will see later.
\end{remark}
\smallskip

\begin{proof}\rm Let $M_{\alpha}:= \sup_{0 \leq
t \leq T}||e^{
tA}||_{L(H_{\alpha})}$.  We fix a positive number
$R$ such that $R\geq
2M_{\alpha}||\xi||_{L^{\infty}(\Omega;H_{\alpha})}$. This implies that $
\sup_{0 \leq t \leq T} ||e^{tA}\xi||_{H_{\alpha}} \leq R/2$ $\P$-a.s.
Moreover, let $L_R$ be such that
$$ |f_0(t,y_1)-f_0(t,y_2)|_H \leq L_R ||y_1-y_2||_{H_{\alpha}}  \quad 0 \leq t\leq T, \quad||y_i||_{H_{\alpha}} \leq R$$

We recall that the space $L^2(\Omega;C([T -
\delta,T];H_{\alpha}))$ is a Banach space endowed with the norm
$Y \rightarrow \big( \E \sup_{t \in [T - \delta, T]}
  ||Y_t||_{H_{\alpha}}^2\big)^{1/2}$. We define
$$\mathbb{K}= \{Y \in L^2 (\Omega; C([T- \delta,T],H_{\alpha})):
 \sup_{t \in [T -\delta,T]}||Y_t||_{H_{\alpha}} \leq R,\quad a.s.\}.$$

It easy to check that $\K$ is a closed subset of $L^2(\Omega; C([T-
\delta,T],H_{\alpha}))$, hence a complete metric space (with the inherited metrics).
We look for a local mild solution $(Y,Z)$ in the space $\mathbb{K}$.
We define a nonlinear operator $\Gamma:\mathbb{K}\rightarrow \mathbb{K}$ as follows:
given $U \in \mathbb{K}$, $Y=\Gamma(U)$ is the first component of the mild solution
$(Y,Z)$ of the equation
\begin{equation} \label{U} Y_t -\int_t^T
e^{(s-t)A}[f_0(s,U_s)ds+f_1(s)]ds + \int_t^T e^{(s-t)A}Z_sdW_s= e^{(T-t)A}\xi
\end{equation}
for $t \in [T - \delta, T]$.
Since $U \in \mathbb{K}$ we have $\P$-a.s.
\begin{equation}\label{fo.f1.limitate}
|f_0(t,U_t) +f_1(t)|_H \leq S(1+ ||U_t||_{H_{\alpha}}^{\gamma}) + C \leq S(1+
R^{\gamma})+ C,
\end{equation}
for all $t$ in $[T - \delta, T]$. Hence $f_0(\cdot,U_{\cdot})
+f_1(\cdot)$ belongs to $L^2(\Omega \times [T - \delta, T];H)$
and, by a result of Hu and Peng \cite{HP}, there exists a unique
pair $(Y,Z) \in L^2(\Omega\times [T- \delta,T];H) \times
L^2(\Omega \times[T- \delta,T];{\mathcal{L}}^2(K;H))$ satisfying
(\ref{U}). Moreover, by taking the conditional expectation given
$\mathcal{F}_t$, $Y$ has the following representation
$$Y_t= {\E}^{\mathcal{F}_t}\big( e^{(T-t)A}\xi + \int_t^T e^{(s-t)A}[f_0(s,U_s)+f_1(s)] ds \big).$$
We will show that $\Gamma$ is a contraction for the norm of $L^2
(\Omega, C([T- \delta,T];H_{\alpha})$ and maps $\mathbb{K}$ into
itself, if $\delta $ is sufficiently small; clearly, its unique
fixed point is the required solution of the BSDE.

We first check the contraction property. Let $U^1,U^2 \in
\mathbb{K}$. Then
$$
\begin{array}{l}
\dis{\Gamma(U^1)_t -\Gamma(U^2)_t=Y_t^1-Y_t^2= {\E}^{\mathcal{F}_t}\big[\int_t^T
e^{(s-t)A}(f_0(s,U_s^1) -f_0(s,U_s^2)\big) ds \big].}\\
\end{array}$$ Let $v(t)= \int_t^T
e^{(s-t)A}\big( f_0(s,U_s^1)-f_0(s,U_s^2)\big)ds$. Then, noting that $v(T)=0$ and
recalling Lemma \ref{phi}, for $t \in [T- \delta,T]$
$$
\begin{array}{lll}
  \dis{||Y_t^1-Y_t^2||_{H_{\alpha}}  =} \\
 = \dis{||{\E}^{\mathcal{F}_t}
  v(t)||_{H_{\alpha}}}
  & \leq & \dis{ {\E}^{\mathcal{F}_t} ||v(t)||_{H_{\alpha}}} \\
 & \leq & \dis{{\delta}^{1 - \alpha}{\E}^{\mathcal{F}_t} ||v||_{C^{(1- \alpha)}([T -
  \delta,T],H_{\alpha})}} \\
  & \leq &
 \dis{G{\delta}^{(1 - \alpha)}\E^{\mathcal{F}_t}||f_0(\cdot,U_{\cdot}^1)-f_0(\cdot,U_{\cdot}^2) ||_{L^{\infty}([T -
 \delta,T],H)}}\\
 & \leq & \dis{G{\delta}^{(1- \alpha)} L_R \E^{\mathcal{F}_t}\sup_{t \in  [T -
 \delta,T]}||U_t^1-U_t^2||_{H_{\alpha}} =: M_t},\\
\end{array}
$$
where $\{M_t, \, t \in[T- \delta,T]\}$ is a martingale. Hence, by Doob's inequality
$$
\begin{array}{lll}
\dis{\E \sup_{t \in [T - \delta,T]}||Y_t^1-Y_t^2||_{H_{\alpha}}^2} & \leq & \dis{\E \sup_{t \in [T
- \delta,T]}|M_t|^2 \leq 2 \E|M_T|^2}= \\
& = & \dis{2G^2L_R^2{\delta}^{2(1-\alpha)} \E\sup_{t \in [T -
\delta,T]}||U_t^1-U_t^2||_{H_{\alpha}}^2.} \end{array}$$
If $\delta  \leq {\delta}_0={2GL_R}^{\frac{-1}{(1-\alpha)}}$, then $\Gamma$ is a
contraction with constant $1/2$.

Next we check that $\Gamma$ maps $\K$ into itself. For each $U \in \mathbb{K}$ and $t\in [T- \delta, T]$ with $\delta \leq
{\delta}_0$ we have
$$
\begin{array}{l}
\dis{\sup_{t \in [T- \delta, T]}||{\Gamma(U)}_t||_{H_{\alpha}}} = \dis{\sup_{t \in [T- \delta,
T]}||Y_t||_{H_{\alpha}}}  \leq   \dis{ \sup_{t \in [T- \delta,
T]}\E^{\mathcal{F}_t}||e^{(T-t)A}
\xi||_{H_{\alpha}}} +\\
 \quad \quad \quad\quad \quad\quad \quad\quad \quad
 \dis{+ \sup_{t \in [T- \delta, T]} {\E}^{\mathcal{F}_t}||\int_t^T
e^{(s-t)A}[f_0(s,U_s)+f_1(s)]ds
||_{H_{\alpha}}}\\
 \quad \quad\quad\quad  \leq  \dis{ R/2 +\sup_{t \in [T- \delta,
T]}{\E}^{\mathcal{F}_t}\int_t^T
||e^{(s-t)A}[f_0(s,U_s)+f_1(s)]||_{H_{\alpha}}} ds\\
\quad \quad \quad\quad \leq  \dis{ R/2 +\sup_{t \in [T- \delta,
T]}{\E}^{\mathcal{F}_t}\int_t^T
||e^{(s-t)A}[f_0(s,U_s)+f_1(s)]||_{D_A(\alpha,1)}}ds, \\
\end{array}
$$
where in the last inequality we have used the fact that $
D_A(\alpha,1)\subset H_{\alpha}$.
Now, by Proposition \ref{2.2.9lun}, and from 4.ii)
and 5., it follows that
$$
\begin{array}{l}
\dis{||e^{(s-t)A}[f_0(s,U_s)+f_1(s)]||_{D_A(\alpha,1)}}  \leq \\
\qquad \leq\dis{||e^{(s-t)A}||_{L(H,D_A(\alpha,1))} |f_0(s,U_s)+f_1(s)|_H }\\
 \qquad \leq  \dis{\frac{C_{\alpha}}{(s-t)^{\alpha}}[S(1+||U_s||_{H_{\alpha}}^{\gamma})+ C]}. \\
\end{array}
$$
Then, since $U \in \K$, we arrive at
$$
\begin{array}{l}
\dis{\sup_{t \in [T- \delta, T]}||{\Gamma(U)}_t||_{H_{\alpha}}} \leq \\
\quad \leq \dis{R/2 + \sup_{t \in [T-
\delta, T]}{\E}^{\mathcal{F}_t}\int_t^T
\frac{C_{\alpha}}{(s-t)^{\alpha}}[S(1+||U_s||_{H_{\alpha}}^{\gamma})+ C]ds}\\
 \quad \leq \dis{R/2 + \sup_{t \in [T- \delta, T]}\int_t^T
\frac{C_{\alpha}}{(s-t)^{\alpha}}[S(1+R^{\gamma})+ C]ds} \\
\quad \leq  \dis{ R/2+ C_{\alpha}S\frac{[(1+R^{\gamma})+C]}{1- \alpha}{\delta}^{1- \alpha}}
,\\
\end{array}
$$
where $C_{\alpha}$ depends on $A$, $\alpha$.
Hence, if $\delta \leq {\delta}_0$ is such that $C_{\alpha}S\frac{[(1+R^{\gamma})+C]}{1-
\alpha}{\delta}^{1- \alpha}$ is less or equal to $R/2$, then $\dis{\sup_{t \in [T- \delta, T]}||{\Gamma(U)}_t||_{H_{\alpha}}} \leq R$.
Due to Lemma \ref{phi}, $\P$-a.s. the function $t \mapsto Y_t -
{\E}^{\mathcal{F}_t}e^{(T-t)A}\xi$ belongs to $C[T-
\delta,T];H_{\alpha})$; moreover, the map $ t \mapsto
{\E}^{\mathcal{F}_t}e^{(T-t)A}\xi$ belongs to $C[T-
\delta,T];H_{\alpha})$, since $\xi$ is a random variable taking
values in $\overline{D(A)}^{H_{\alpha}}$. Therefore, $\P$-a.s.
$Y_{\cdot} \in C([T-\delta,T];H_{\alpha})$ and
$\Gamma$ maps $\mathbb{K}$ into itself and
has a unique fixed point in $\mathbb{K}$.
 \end{proof}

\begin{remark} \label{b1} \rm By Lemma \ref{phi}, using properties of analytic
semigroups, it can be proved that for every fixed $\omega$ the range of the map
$\Gamma$ is contained in $C^{1 - \beta}([T-\delta, T - \epsilon];D_A(\beta,1))$ for
every $\epsilon \in (0,\delta)$, $\beta \in [0,1]$. \end{remark}
\medskip
\subsection{Global existence}\label{GE}
Now we are able to prove a global existence theorem for the solution of the equation
(\ref{e}), using all the results presented above.

\begin{theorem}\label{cap4.T2}If Hypothesis \ref{H1}
is satisfied, the equation (\ref{e}) has a unique mild solution $(Y,Z) \in
L^2(\Omega;C([0,T], H_{\alpha})) \times L^2(\Omega \times [0,T]);\mathcal{L}^2(K;H))$.
\end{theorem}
\begin{proof} \rm By Theorem \ref{el} equation (\ref{e}) has a unique mild solution
$(Y^1,Z^1) \in L^2(\Omega;C([T - \delta_1,T], H_{\alpha})) \times L^2(\Omega \times [T-
\delta_1,T]);\mathcal{L}^2(K;H))$ on the interval $[T- \delta_1,T]$, for some $\delta_1>0$.
 By Proposition \ref{eH} we know that there exists a constant $C_1$ such that $\P$-a.s.
 \begin{equation}\label{d1H}|Y_{T-\delta_1}|_H \leq C_1.\end{equation}
 We recall that the constant $C_1$ depends only
on $|\xi|_{L^{\infty}( \Omega;H)}$ and on the constants $S$ of 4.ii) and $C$ of 5. and is independent of $\delta_1$.
Moreover, by Proposition \ref{eHa}, there exists a constant $C_2$
such that $\P$-a.s.
\begin{equation}\label{C2}||Y_{T - \delta_1}||_{L^{\infty}(\Omega, D_A(\theta,\infty))} \leq C_2
\frac{1}{\delta_1^{\theta - \alpha}},\end{equation}
with $C_2$ depending on the operator $A$, $||\xi||_{L^{\infty}(\Omega, H_{\alpha})}$,
$\theta$, $\alpha$, $C_1 $. This implies that $Y_{T- \delta_1}$
 belongs to $L^{\infty}(\Omega;H_{\alpha})$ and it can be taken as final value for the problem
\begin{equation}\begin{array}{l}\label{Y^2}Y_t -\int_t^{T- \delta_1} e^{(s-t)A}[f_0(s,Y_s)ds+f_1(s)]ds + \int_t^{T- \delta_1} e^{(s-t)A}Z_sdW_s=\\
\qquad \qquad \qquad \qquad = e^{(T-\delta_1 -t)A}Y_{T-\delta_1}
\end{array}
  \end{equation}
  on an interval $[T - \delta_1 - \delta_2, T-\delta_1]$, for some $\delta_2>0$.
  As in the proof of Theorem
\ref{el}, we fix a positive number $R_2$ such that
$$R_2=
2M_{\alpha}\frac{C_2}{{\delta_1}^{\theta - \alpha}} \geq 2M_{\alpha}||Y_{T -
\delta_1}||_{L^{\infty}(\Omega, D_A(\theta,\infty))}.$$

By Theorem \ref{el} there exists a pair of progressively
measurable processes $(Y^2,Z^2)$ in $L^2(\Omega;C([T -
\delta_1-\delta_2,T- \delta_1];H_{\alpha}))\times L^2(\Omega
\times [T- \delta_1-\delta_2,T-\delta_1];{\mathcal{L}}^2(K,H))$
which solves (\ref{Y^2}) on the interval $[T -\delta_1 -
\delta_2,T - \delta_1]$ where $\delta_2$ depends on the operator
$A$, $\alpha$, $R_2$. We note that the continuity in $T- \delta_1$
of $Y^2$ follows from the fact that $Y_{T- \delta_1}$ takes values
in $D_A(\alpha,1)$ (see Remark \ref{b1}), so that $Y_{T-
\delta_1}$ takes values in $\overline{D(A)}^{H_{\alpha}}$. Now,
the process $Y_t$ defined by $Y_t^1$ on the interval $[T
-\delta_1, T]$ and by $Y_t^2$ on $[T -\delta_1 - \delta_2,T -
\delta_1]$ belongs to $L^2(\Omega;C([T -
\delta_1-\delta_2,T];H_{\alpha})) $ and it easy to see that it
satisfies (\ref{e}) in the whole interval $[T -\delta_1 -
\delta_2,T]$. Consequently, by Proposition \ref{eH}, $\P$-a.s.,
$|Y_{T -\delta_1 -\delta_2}|_H \leq C_1$ with $C_1$ the constant
in (\ref{d1H}), and by (\ref{C2})
\begin{equation}\label{d2}
\begin{array}{l}
\dis{||Y_{T -\delta_1 -\delta_2}||_{L^{\infty}(\Omega,D_A(\theta, \infty))}}  \leq
\dis{\frac{C_2}{(\delta_1 + \delta_2)^{\theta- \alpha}}} \leq \dis{\frac{C_2}{{\delta_1}^{\theta- \alpha}}},
\end{array}
\end{equation}
where $C_2$ is the same constant as in (\ref{C2}).
Again, $Y_{T -\delta_1 -\delta_2}$ can be taken as initial value for problem
\begin{equation}\label{Y^3}\begin{array}{l}
\dis{Y_t -\int_t^{T -\delta_1 -\delta_2} e^{(s-t)A}[f_0(s,Y_s)ds+f_1(s)]ds + \int_t^{T
-\delta_1 -\delta_2}
e^{(s-t)A}Z_sdW_s=}\\
\quad \quad \quad \quad \quad \quad \quad \quad \quad\dis{ e^{(T -\delta_1 -\delta_2
-t)A}Y_{T-\delta_1-\delta_2}}\\
\end{array}
  \end{equation}
on the interval $[T -\delta_1 - \delta_2 -\delta_3,T - \delta_1 -\delta_2]$, where
$\delta_3$ will be fixed later. In this case, by (\ref{d2}), we can choose
$$R_3=
R_2=2M_{\alpha}\frac{C_2}{{\delta_1}^{\theta - \alpha}} \geq 2M_{\alpha}||Y_{T -
\delta_1- \delta_2}||_{L^{\infty}(\Omega,D_A(\theta, \infty))}$$ and prove that there
exists a unique mild solution $(Y^3,Z^3)$ of
(\ref{Y^3}) on the interval $[T -\delta_1 - \delta_2 -\delta_3,T - \delta_1 -\delta_2]$,
with $\delta_3= \delta_2$ . So we extend the solution to $[T -\delta_1- 2\delta_2,T]$.
Proceeding this way we prove the global existence to (\ref{e}) on $[0,T]$.
\end{proof}

\section{The general case}\label{3}

We can now study the equation:
\begin{equation}\label{pbdef}
Y_t -\int_t^T e^{(s-t)A}[f_0(s,Y_s)+f_1(s,Y_s,Z_s)]ds + \int_t^T e^{(s-t)A}Z_sdW_s=
e^{(T-t)A}\xi
\end{equation}

We require that the function $f_1$ satisfy the following assumptions:
\begin{hypothesis}\label{H2}
\end{hypothesis}
\begin{description}
  \item [1.] there exists $K \geq 0$ such that $\P$-a.s.

  $\quad\quad|f_1(t,y,z)-f_1(t,y^{'},z^{'})|_H\leq K|y-y^{'}|_H +
  K||z-z^{'}||_{{\mathcal{L}}^2(K;H)},$

for every $t \in [0,T], y,y^{'}\in H, z,z^{'}\in {\mathcal{L}}^2(K;H),$
  \item [2.] there exists $C \geq 0$ such that $\P$-a.s.
  $|f_1(t,y,z)|_H \leq C,$
  for every $t\in [0,T], y \in H, z \in {\mathcal{L}}^2(K;H)$.
  \end{description}

\begin{theorem} \label{cap4.T3} If Hypotheses \ref{H1} and \ref{H2} hold, then equation
(\ref{pbdef}) has a unique solution in $L^2(\Omega; C( [0,T];H_{\alpha})) \times
L^2(\Omega\times[0,T];{\mathcal{L}}^2(K;H))$.
\end{theorem}

\begin{proof} \rm Let $\mathbb{M}$ be the space of
progressive processes $(Y,Z)$ in the space $L^2(\Omega \times [0,T];H)\times
L^2(\Omega\times[0,T];{\mathcal{L}}^2(K;H))$ endowed with the norm
$$|||(Y,Z)|||_{\beta}^2={\E}\int_0^T e^{\beta s}(|Y_s|_H^2+
||Z_s||_{{\mathcal{L}}^2(K;H)}^2)ds,$$ where $\beta$ will be fixed later.
We define $\Phi:\mathbb{M}\rightarrow \mathbb{M}$ as follows: given
$(U,V)\in \mathbb{M}$, $(Y,Z)= \Phi(U,V)$ is the unique solution on the interval $[0,T]$ of the equation
$$Y_t -\int_t^T e^{(s-t)A}[f_0(s,Y_s)ds+f_1(s,U_s,V_s)]ds + \int_t^T
e^{(s-t)A}Z_sdW_s= e^{(T-t)A}\xi.$$ By Theorem
\ref{cap4.T2} the above equation has a unique mild solution
$(Y,Z)$ which belongs to $L^2(\Omega; C( [0,T];H_{\alpha})) \times
L^2(\Omega\times[0,T];{\mathcal{L}}^2(K;H))$. Therefore $
\Phi(\mathbb{M}) \subset \mathbb{M}$. We will show that $\Phi$ is
a contraction for a suitable choice of $\beta$; clearly, its
unique fixed point is the required solution of (\ref{pbdef}).
 We take
another pair $(U^{'},V^{'}) \in \mathbb{M}$ and apply Proposition 3.1 in \cite{CF}
 to the difference of two equations. We obtain
$$ \begin{array}{l}
\dis{\E\int_0^T e^{\beta s}\big[ \beta|Y_t^1-Y_t^2|_H^2 +\|Z_s^1
-Z_s^2\|_{{\mathcal{L}}^2(K;H)}^2ds\big]} \leq\\
  \leq \dis{2\E\int_0^T e^{\beta s}
<f_0(s,Y_s^1)+f_1(s,U_s^1,V_s^1)+}\\
 \quad  \quad   \dis{-f_0(s,Y_s^2)-f_1(s,U_s^2,V_s^2),Y_s^1-Y_s^2>_Hds} \\
   \leq \dis{2\E\int_0^T e^{\beta s}K(|U_s^1-U_s^2|_H +
||V_s^1
-V_s^2||_{{\mathcal{L}}^2(K;H)})|Y_s^1-Y_s^2|_Hds} \\
 \leq \dis{\E\int_0^T e^{\beta s}(|U_s^1-U_s^2|_H^2 +
||V_s^1 -V_s^2||_{{\mathcal{L}}^2(K;H)}^2)/2
+ 4K^2|Y_s^1-Y_s^2|_H^2ds},\\
\end{array}
$$
where we have used 4.iv) of Hypothesis 2 and 1. of Hypothesis
\ref{H2}.
Choosing $\beta =4K^2 +1$, we obtain the required contraction property. \end{proof}

\section{Applications}\label{4}

In this section we present some backward stochastic partial differential problems which
can be solved with our techniques.

\subsection{The reaction-diffusion equation}

Let $D$ be an open and bounded subset of $\mathbb{R}^n$ with a
smooth boundary $\partial D$. We choose $K=L^2(D)$. This choice
implies that $dW_t/dt$ is the so-called "space-time white noise".
Moreover, since Hilbert-Schmidt operators on $L^2(D)$ are
represented by square integrable kernels, the space
$\mathcal{L}^2(L^2(D),L^2(D))$ can be identified with $L^2(D
\times D)$. We are given a complete probability space $(\Omega,
\mathcal{F},\P)$ with a filtration $(\mathcal{F}_t)_{t \in [0,T]}$
generated by $W$ and augmented in the usual way. 
Let us consider a non symmetric bilinear, coercive
continuous form $a :H_0^1(D)\times H_0^1(D)\rightarrow \mathbb{R}$ defined by
$a(u,v):=-\int_{D} \sum_{i,j}^n a_{ij}(x)D_iu(x)D_jv(x) dx,$
where the coefficients $a_{ij}$ are Lipschitz continuous and there exists $\alpha
>0$ such that $\sum_{i,j=1}^n  a_{ij}(x)\xi_i \xi_j \geq \alpha|\xi|^2$ for every $x \in \overline{D}$, $\xi \in {\mathbb{R}}^n$. Let $A$ be the
operator associated with the bilinear form $a$ such that $ <Au,v>_{L^2(D)}=a(u,v)$, $v \in H_0^1(D)$ and $u \in D(A)$. It is known that, in this case, $D(A)=H^2(D) \cap H_0^1(D)$, where $H^2(D)$ and $H_0^1(D)$ are the usual Sobolev spaces. 

We consider for $t \in [0,T]$ and $x \in D$ the backward
stochastic problem written formally
\begin{equation} \label{heat.eq}
\left \{ \begin{array}{lll}
 \dis{\frac{\partial Y(t,x)}{\partial t} = A Y(t,x) +r(Y(t,x))} & + & \dis{
 g(t,Y(t,x),Z(t,x),x)+}\\
& + & \dis{Z(t,x)\frac{\partial W(t,x)}{\partial t} }   \\
& &       \quad \quad  \quad \quad\quad \textrm{on $\Omega \times [0,T]\times \bar{D}$}\\
\dis{Y(T,x)=\xi (x)} & & \quad \quad  \quad \quad\quad  \textrm{on $\Omega \times \bar{D}$}\\
\dis{Y(t,x)=0} & &  \quad \quad  \quad \quad \quad  \textrm{on $\Omega \times [0,T] \times \partial D$}\\
\end{array} \right.
\end{equation}
We suppose the following.
\begin{hypothesis}\label{ip.g}
\end{hypothesis}
 \rm
\begin{description}
  \item [1.] $r :\R \rightarrow \R$ is a continuous, increasing and locally Lipschitz function;
  \item [2.] $r$ satisfies the
  following growth condition: $|r(x)| \leq S(1 +|x|^{\gamma}) \quad \forall x \in \R$ for some $\gamma >1$;
  \item [3.] $g$ is a measurable real function defined on $[0,T] \times \R \times L^2(D \times D) \times D$ and there exists a constant $K>0$ such that
  $$|g(t,y_1,z_1,x)-g(t,y_2,z_2,x)| \leq K(|y_1 - y_2| + ||z_1 -z_2||_{L^2(D \times D)})$$ for all $t \in[0,T]$, $y_1,y_2 \in \R$, $z_1,z_2 \in L^2(D)$, $x \in D$;
  \item [4.] there exists a real function $h$ in $L^2(D \times D)$ such that $\P$-a.s. $|g(t,y,z,x)| \leq K_1 h(x)$ for all $t \in [0,T]$, $y \in \R$, $z \in L^2(D)$, $x \in D$;
  \item [5.] $\xi$ belongs to $L^{\infty}(\Omega;H^2(D) \cap
  H_0^1(D))$.
\end{description}

We define the operator $A$ by
$(Ay)(x)= Ay(x)$  with domain $D(A)=  H^2(D) \cap H_0^1(D)$.
We set $f_0(t,y)(x)= -r(y(t,x))$ for $t \in [0,T]$, $x \in D$ and $y$ in a suitable subspace
of $H$ which will be determined below. For $t \in [0,T]$, $x \in D$,
$y \in L^2(D)$, $z \in L^2(D \times D)$ we define $f_1$ as the operator
$f_1(t,y,z)(x)= -g(t,y(t,x),z(t,x),x)$.
Then problem (\ref{heat.eq}) can be written in abstract way as

$\quad \quad  dY_t = -AY_tdt - f_0(t,Y_t)dt - f_1(t,Y_t,Z_t)dt +
Z_t dW_t, \quad
Y_T=\xi . $

Under the conditions in Hypothesis \ref{ip.g}, the
assumptions in Hypotheses \ref{H1}, \ref{H2} are satisfied.
The operator $A$ is a closed operator in $L^2(D)$ and it is the
infinitesimal generator of an analytic semigroup in $L^2(D)$ satisfying
$\|e^{tA}\|_{\mathcal{L}(H)} \leq 1 $ (see \cite{T}, Chapter 3). In particular, by Lumer-Philips theorem, $A$ is dissipative.
The non linear function $f_0(t,\cdot): L^{2\gamma}(D) \rightarrow L^2(D)$,
$ y \mapsto - r(y)$
is locally Lipschitz. We look for a space of class $J_{\alpha}$ between $H$ and $D(A)$
where $f_0$ is well defined and locally Lipschitz. It is well known (see \cite{Tr}) that
the fractional order Sobolev space $W^{\beta,2}(D)$ is of class $J_{\beta/2}$ between $L^2(D)$ and $H^2(D)$ for every
$\beta \in(0,2)$. Hence the space $H_{\alpha}$ defined by
$H_{\alpha}= W^{\beta,2}(D)$ if $\beta<1$, by $W^{\beta,2}(D) \cap H_0^1(D)$ if $\beta \geq 1$
is of class $J_{\beta/2}$ between H and $D(A)$. Moreover the restriction of $A$ on
$H_{\alpha}$ is a sectorial operator (\cite{Tr}).
By the Sobolev embedding theorem, $W^{\beta,2}$ is contained in $L^q(D)$ for all $q$ if $\beta\geq \frac{n}{2}$,
and in $L^{2n/(n-2\beta )}(D)$ if $\beta <\frac{n}{2}$.
 If we choose $\beta \in(0,2)$ we have $W^{\beta,2}(D) \subset
L^{2\gamma}(D)$ for $n<4 \frac{\gamma}{\gamma-1}$. It is clear that $f_0$ is locally Lipschitz with
respect to $y$ from $H_{\alpha}$ into $H$.
It is easy to verify that $f_0$ satisfies 4.ii) of Hypothesis \ref{H1} with $\gamma =2n+1$ and that it is dissipative with constant $\mu=0$.
The function $f_1$ is Lipschitz uniformly with
respect to $y$ and $z$ and it is bounded.
The final condition $\xi$ takes values in $\overline{D(A)}^{H_{\alpha}}$ and belongs to
 $L^{\infty}(\Omega;H_{\alpha})$.
Hence we can apply the global existence theorem and state that the above problem
has a unique mild solution $(Y,Z)\in L^2 (\Omega; C( [0,T];H_{\alpha}))\times L^2(\Omega
\times [0,T];{\mathcal{L}}^2(K,H))$.

\subsection{A spin system}\label{spin}
Let $\mathbb{Z}$ be the one-dimensional lattice of integers. Its elements will be
interpreted as atoms. A \emph{configuration} is a  real function $y$ defined
on $\mathbb{Z}$. The value $y(n)$ of the configuration $y$ at the point $n$ can be
viewed as the state of the atom $n$.

We consider an infinite system of equations
\begin{equation} \label{spinsys}
dY_t^n=-a_nY_t^n dt + \sum_{|n-j| \leq 1} V(Y_t^n -Y_t^j)dt +Z_t^ndW_t^n \quad n \in \mathbb{Z},\quad 0 \leq t\leq T
\end{equation}
$\quad Y_n(T)= \xi_n   \quad n \in \mathbb{Z},$

where $Y^n$ and $Z^n$ are real processes, and $V: \R \rightarrow \R$.

Let $l^2(\Z)$ be the usual Hilbert space of square summable sequences. To study system
(\ref{spinsys}) we apply results of previous sections. To fit our assumption in
Hypotheses \ref{H1} and \ref{H2}, we suppose the following

\begin{hypothesis}

\label{ip.spin} \rm
\end{hypothesis}
\begin{description}

  \item [1.] $W^n$, $ n \in \Z$ are independent standard real Wiener processes;

  \item [2.] $a=\{a_n \}_{n \in \Z}$ is a sequence of nonnegative real numbers;
  \item [3.] $\xi =\{\xi_n \}_{n \in \Z}$
is a random variable belonging to $L^{\infty}(\Omega,l^2(\Z))$;
  \item [4.] the function $V : \R \rightarrow \R$ is defined by
$V(x)= x^{2k+1}\quad k \in \mathbb{N}.$
\end{description}

We will study system (\ref{spinsys}) regarded as a backward stochastic evolution
equation for $t \in [0,T]$
\begin{equation} \label{spineq} dY_t=(AY_t + f_0(t,Y_t))dt + Z_tdW_t, \quad
   Y_T= \xi
 \end{equation}
on a properly chosen Hilbert space $H$ of
functions on $\mathbb{Z}$.

To reformulate problem (\ref{spinsys}) in the abstract form
(\ref{spineq}), we set $K=H=l^2(\Z)$. We set $W_t= \{W_t^n\}_{n
\in \mathbb{Z}}, t \in [0,T]$. By 1. of Hypothesis \ref{ip.spin},
$W$ is a cylindrical Wiener process in $H$ defined on $(\Omega,
\mathcal{F}, P)$. We define the operator $A$ by

$\quad \quad A(y)=(a_n y_n )_n, \quad D(A)=\{ y \in l^2(\mathbb{Z}) \mbox{ such that } \sum_{n \in \mathbb{Z}} a_n^2y_n^2 <
\infty \}.$

It is easy to prove that $A$ is  a self-adjoint operator in $l^2(\Z)$, hence the infinitesimal generator of a sectorial
semigroup. The coefficient $f_0$ is given by
$(f_0(t,y))_n=(V(y_{n+1} -y_n) +V( y_{n-1}- y_n)), \quad t \in [0,T], y \in D(f_0)$ where
$D(f_0)= \{ y \in l^2(\mathbb{Z}) \mbox{ such that } \sum_{n \in \mathbb{Z}} |x_{n+1} -x_n|^{2(2k+1)}< + \infty \}.$
Under Hypothesis \ref{ip.spin}, $A,f_0, \xi$ satisfy Hypotheses \ref{H1} and \ref{H2}.
We observe that in this case the domain of $f_0$ is the whole space $H$: if $y \in
l^2(\mathbb{Z})$ then
$$
\begin{array}{l}
\dis{\{ \sum_{n \in \mathbb{Z}} |y_{n+1} -y_n|^{2(2k+1)}\}^{\frac{1}{2(2k+1)}}} \leq
\dis{ \{\sum_{n \in \mathbb{Z}} |y_{n+1} -y_n|^2 \}^{\frac{1}{2}}}
  \leq \dis{2||y||_{l^2(\mathbb{Z})}}.\\
 \end{array}$$
 Consequently, we can take $H_{\alpha}$ with $\alpha=0$, i.e. $H_0=H$.
The function $f_0$ is dissipative. Namely
$$
\begin{array}{l}
\dis{<f_0(t,y) - f_0(t,y'),y - y'>_{l^2(\mathbb{Z})} } =\\
\qquad
=\dis{\sum_{n \in \mathbb{Z}} \{[(y_{n+1} -y_n)^{(2k+1)}+ (y_{n-1} -y_n)^{(2k+1)}] +}\\
\qquad \dis{+ [(y'_{n+1} -y'_n)^{(2k+1)} + (y'_{n-1} -y'_n)^{(2k+1)}])[y_n -y'_n]}=\\
\dis{=-\sum_{n \in \mathbb{Z}} [(y_{n+1} -y_n)^{(2k+1)}- (y'_{n+1}
-y'_n)^{(2k+1)}][(y_{n+1} -y_n)-(y'_{n+1} -y'_n)]}\\
\end{array}
$$
and the last term is negative.
Moreover, $f_0$ satisfies 4.ii) of Hypothesis \ref{H1} with $\gamma=2k+1$.
The map $f_0$ is also locally Lipschitz from $H$ in to $H$.
Then by Theorem \ref{cap4.T3}, problem (\ref{spineq}) has a unique mild solution $(Y,Z)$
which belongs to $ L^2(\Omega,C([0,T];H)) \times L^2(\Omega \times
[0,T];\mathcal{L}^2(K,H))$.

\section{Appendix}
This section is devoted to the proof of Lemma \ref{Gronwall2}. Assume first that $\beta=1$. Using recursively the inequality $\dis{ U_t  \leq    a(T-t)^{-\alpha} +b \int_t^T \E^{\mathcal{F}_t}U_sds }$
we can easily prove that

$U_t \leq  $

$\quad \,\dis{\leq a (T-t)^{-\alpha} +\int_t^T
a\sum_{k=1}^{n-1}b^k\frac{(r-t)^{k-1}}{(k-1)!}\frac{1}{(T-r)^{\alpha}}dr +}\\
\qquad \qquad \qquad \qquad \dis{ \qquad \qquad    +b\E^{\mathcal{F}_t}\int_t^T \frac{(b(r-t))^{n-1}}{(n-1)!}U_rdr.}
$

The last term in the above inequality
tends to zero as $n$ tends to infinity for each $t
$ in the interval $[0,T]$. Thus
$$
\begin{array}{lll}
\dis{U_t} & \leq  & \dis{a (T-t)^{-\alpha} + a\sum_{k=1}^{\infty}\int_t^T
\frac{b^{k}(T-t)^{k-1}}{(k-1)!}\frac{1}{(T-r)^{\alpha}}dr}\leq \\
& \leq  & \dis{a (T-t)^{-\alpha} + abe^{b(T-t)}\int_t^T\frac{1}{(T-r)^{\alpha}}dr} \leq\\
 & \leq  & \dis{a (T-t)^{-\alpha} + abe^{b(T-t)}\frac{1}{1 -\alpha}{(T-t)^{1-\alpha}}} \leq  \dis{a (T-t)^{-\alpha}M}\\
\end{array}
$$
where $M=1 + be^{bT}\frac{1}{1 -\alpha}{T}$.

In the case $\beta \neq 1$ a similar proof can be given, based on recursive use of the inequality
$\quad \dis{U_t}  \leq  \dis{a (T-t)^{-\alpha} +b \int_t^T
(s-t)^{\beta-1}\E^{\mathcal{F}_t}U_sds}.$

{\bf Acknowledgments}:

I wish to thank Giuseppe Da Prato for hospitality at the Scuola
Normale Superiore in Pisa, suggestions and helpful discussions. I
would like to express my gratitude to Marco Fuhrman: I am indebted
to him for his precious help and encouragement. Special thanks are
due to Alessandra Lunardi, who gave me valuable advice and
support.

\end{document}